\documentclass{article}

\usepackage[all]{xy}
\usepackage{mathrsfs}
\usepackage{amssymb}
\usepackage{amsthm}

\newtheorem{thm}{Theorem}
\newtheorem{lem}[thm]{Lemma}
\newtheorem{prop}[thm]{Proposition}
\newtheorem{cor}[thm]{Corollary}

\newtheorem{defn}{Definition}
\newtheorem{exa}{Example}

\newtheorem{rem}{Remark}

\begin{document}

\date{}

\title{On linear differential equations with reductive Galois group}

\author{Camilo Sanabria Malag\'on\footnote{partially supported by NSF grant CCF 0901175 and CCF 0952591}}

\maketitle

\begin{abstract}
Given a connection on a meromorphic vector bundle over a compact
Riemann surface with reductive Galois group, we associate to it a
projective variety.  Connections such that their associated
projective variety are curves can be classified, up to projective
equivalence, using ruled surfaces. In particular, such a
meromorphic connection is the pullbacks of a Standard connection.
This extend a similar result by Klein for second-order ordinary
linear differential equations to a broader class of equations.
\end{abstract}

\section*{Introduction}

In ~\cite{baldassarri}, ~\cite{baldassarri2} Baldassari and Dwork
give a contemporary formulation of a result known to Klein
~\cite{klein},~\cite{klein2} on second order ordinary linear
differential equations with algebraic solutions. The result is
most easily stated in terms of projective equivalence (cf. Definition \ref{defnprojequiv}).

\begin{thm}
If an ordinary second order linear differential equation with rational coefficients has
finite projective Galois group, it is projectively equivalent to a
pullback, by a rational map, of a hypergeometric equation.
\end{thm}

The collection of hypergeometric equations appearing in Klein's
theorem can be classified using Schwartz triples as they
correspond to Galois coverings of the Riemann Sphere by another
Riemann Sphere. For each finite group in $PGL_2(\mathbb{C})$ there
is one hypergeometric equation.

In broad terms the argument is as follows: let $y_1$ and $y_2$ be
two $\mathbb{C}$-linearly independent solutions for an equation
satisfying the hypothesis of the theorem and let $G\subseteq
GL_2(\mathbb{C})$ be its differential Galois group. The covering
and the pullback from the theorem arise by taking the composition
\[
t:\mathbb{P}^1(\mathbb{C})\longrightarrow\mathbb{P}^1(\mathbb{C})/G\simeq\mathbb{P}^1(\mathbb{C})
\]
given by:
\[
\xymatrix{
 &  & (y_1(x):y_2(x))\ar@{|->}[d] & \mathbb{P}^1(\mathbb{C})\ar[d] \\
(x:1)\ar@{|->}[urr]\ar@{|-->}[rr]_{t} & & (y_1:y_2)\cdot G & \mathbb{P}^1(\mathbb{C})/G
}
\]
where $G$ is acting by M\"oebius transformations on
$\mathbb{P}^1(\mathbb{C})$. If $t$ is an isomorphism then the
equation is hypergeometric; if not, then it gives the rational
pullback map. When $t$ is an isomorphism, the hypergeometric functions $(y_1:y_2)$ locally give sections of the covering  $\mathbb{P}^1(\mathbb{C})\rightarrow \mathbb{P}^1(\mathbb{C})/G$.

An extension of this result for the third order case was obtained
by M. Berkenbosch ~\cite{ber}, where he also gives an algorithmic
implementation of Klein's result simplifying Kovacic's
algorithm ~\cite{kovacic}. Berkenbosch introduces the concept of
``standard equations'' in order to state his generalization. A
standard equation is an ordinary linear differential equation
which is minimal in the sense that any other ordinary linear
differential equation with finite projective Galois group must be
a pullback thereof. The collection of standard equations is
infinite (even for a fixed group $G$!) and so far it lacked of structure; therefore classifying these equations has not been done yet.

The purpose of this article is to formulate an extension of
Klein's theorem for equations of arbitrary order that also covers
many non-algebraic cases and to give a classification of the
standard equations using ruled surfaces. We treat the problem in
terms of differential modules and connections.

Our main tool in achieving this extension is Compoint's theorem.
This result gives a very concrete description of the maximal
differential ideal involved in the construction of a
Picard-Vessiot extension for the connection. In the first part of
this article we introduce the geometric concepts involved. In the
second we give an algebraic interpretation of these concepts. In
the third and final part, we prove the generalization and introduce the
classifying ruled surfaces. In the appendix we study some
of the consequences for the algebraic case.

We finish this introduction by studying and re-interpreting the original mapping considered by Klein \cite{baldassarri, baldassarri2, klein, klein2} and Fano \cite{fano} so that our classifying ruled surface becomes apparent.

\subsection*{Motivation: The algebraic case}

Consider an irreducible homogeneous ordinary linear differential equation
\[
L(y)=\frac{d^n}{dz^n}y+a_{n-1}\frac{d^{n-1}}{dz^{n-1}}y+\ldots+a_1\frac{d}{dz}y+a_0y=0,
\]
with $a_{n-1},\ldots,a_1,a_0\in\mathbb{C}(z)$. Let $U\subseteq\mathbb{C}$ be an open set avoiding the set of singularities $S$ of $L(y)=0$ and admitting $n$ $\mathbb{C}$-linearly independent solutions
\[
y_i:U\longrightarrow \mathbb{C},\quad i\in\{0,1,\ldots,n-1\}.
\]
These $n$ solutions define an analytic map
\begin{eqnarray*}
U &\longrightarrow & \mathbb{P}^{n-1}(\mathbb{C})\\
z &\longmapsto     & (y_0(z):y_1(z):\ldots:y_{n-1}(z)).
\end{eqnarray*}
By analytic continuation, we can extend this map to a multi-valued map on all $\mathbb{C}\setminus S$. The map is multi-valued up to monodromy, therefore, as the differential Galois group $G\subseteq GL_n(\mathbb{C})$ of the equation contains its monodromy group, our analytic continuation becomes single-valued if we post-compose it by the quotient map
\[
\mathbb{P}^{n-1}(\mathbb{C})\longrightarrow\mathbb{P}^{n-1}(\mathbb{C})/G,
\]
where the linear group $G$ is acting by projective linear automorphisms. We denote the resulting composition by $t: \mathbb{C}\setminus S\rightarrow \mathbb{P}^{n-1}(\mathbb{C})/G$.

Let us assume for the rest of this introduction that $G$ is finite and so it coincides with the monodromy group of $L(y)=0$. In particular, under such assumption, the $y_i$'s are algebraic over $\mathbb{C}(z)$.

As we pointed out, in the case where $n=2$, $t$ is injective implies that $L(y)=0$ is a
hypergeometric equation. Berkenbosch's idea was to consider the case when $n=3$ calling $L(y)=0$ a standard equation if $t$ is injective \cite{ber}. So with Berkenbosch's approach the standard equation would no longer be identified with the quotient map $\mathbb{P}^{2}(\mathbb{C})\rightarrow\mathbb{P}^{2}(\mathbb{C})/G$, but by the pair: the quotient map together with the Zariski closure of the image of the multi-valued analytic map
\begin{eqnarray*}
\mathbb{C}\setminus S &\longrightarrow & \mathbb{P}^{2}(\mathbb{C})\\
z &\longmapsto     & (y_0(z):y_1(z):y_2(z)).
\end{eqnarray*}
To obtain the quotient map one just needs the know the Galois group of $L(y)=0$, to obtain the genus of the closure of the image one relies on the exponents of the singularities $S$ (see Appendix). The drawback of this approach is that although the genus of the image does not depend on the choice of the $y_i$'s, the homogeneous polynomial $P[X_0,X_1,X_2]$ vanishing at the triple $(y_0,y_1,y_2)$ does, making hard to identify exactly which standard equation it corresponds to, complicating the identification of the map $t$ and, therefore, any algorithmic implementation of the result.

In this paper we will deviate from this approach and work instead with line bundles and ruled surfaces. First, note that since $t$ is algebraic then we can extend its domain from $\mathbb{C}\setminus S$ to $\mathbb{P}^1(\mathbb{C})$. Secondly, the map $t$ defines a line bundle $\mathscr{L}$ (cf. \cite[Section II.7]{hart}). Thirdly, we can obtain a second line bundle $\mathscr{L}'$ by working with the map $t_1$ defined similarly to $t$ but using the $n$-tuple $(y_0',y_1',\ldots,y_{n-1}')$ instead of $(y_0,y_1,\ldots,y_{n-1})$. Finally, the projective space bundle defined by the rank-2 vector bundle $\mathscr{L}\oplus\mathscr{L}'$ will characterize the standard equation. Those interested in seeing which ruled surfaces correspond to the second order equations can jump to the Appendix and then comeback to read the rest of the paper.

To deal with the non-algebraic cases we will rely on matrix differential equations instead of linear differential equations. So instead of using the map given by $n$ linearly independent solutions, we will use the map given by the full system of solutions. In order to obtain a coordinate-free description we will rely on the concept of connections.

\section{Geometric considerations}

We first establish notation. $X$ denotes a compact Riemann
surface, $k:=\mathbb{C}(X)$ is the associated field of meromorphic
functions, and we consider a rank $n$ meromorphic vector bundle
over $X$ induced by a holomorphic vector bundle $\Pi: E\rightarrow
X$ (cf.~\cite{morales}). By abuse of notation we will use $\Pi:
E\rightarrow X$ to describe both: the holomorphic and the induced
meromorphic vector bundle.

The sheaf of meromorphic sections of $\Pi$ will be denoted by
$\mathscr{E}$, and the sheaf of meromorphic functions over $X$ by
$\mathscr{M}$. The concept of the sheaf $\mathscr{E}$ of
meromorphic sections of a holomorphic vector bundle can be found
in full detail in~\cite{Forster}. The sheaf of meromorphic
$1$-forms, and the sheaf of meromorphic tangent fields, will be
denoted by $\Omega^1_\mathscr{M}$ and $\mathscr{T}X$ respectively.
The sheaf of differential forms $\Omega^1_\mathscr{M}$  is the
meromorphic dual of $\mathscr{T}X$ (cf.~\cite{morales}). Given an
$f\in k$ there is a global meromorphic differential form
$df\in\Omega^1_\mathscr{M}(X)$ defined as follows:
\begin{eqnarray*}
df: \mathscr{T}X(X) & \longrightarrow & k\\
v & \longmapsto & df(v): p \mapsto v_p(f).
\end{eqnarray*}

Any global tangent field $v\in \mathscr{T}X(X)$  induces a
derivation in $k$, i.e. the map
\begin{eqnarray*}
v:k & \longrightarrow & k\\
  f & \longmapsto & v(f):p\mapsto v_p(f)
\end{eqnarray*}
is additive and satisfies the Leibniz rule:
\begin{eqnarray*}
v(f+g) & = & v(f)+v(g)\\
v(fg)  & = & v(f)g+fv(g), \quad \forall (f,g)\in k^2.
\end{eqnarray*}

Once we fix $v$, the field $k$ together with the derivation
defined by $v$ is a differential field. The field of complex
numbers $\mathbb{C}$ can be identified with a subfield of $k$ by
regarding the complex numbers as constant functions. With this
identification in mind we see that the kernel of $v$, known as the
\emph{constants} of the differential field, is $\mathbb{C}$
provided $v\ne 0$. For let $x\in k$ be such that $x$ is not a
constant, then $k$ is an algebraic extension of $\mathbb{C}(x)$.
Furthermore the derivation $d/dx$ of $\mathbb{C}(x)$ extends
uniquely to a no-new-constants derivation $v_x$ of $k$
\cite[Exercises 1.5.3]{SvdP}, that is $\{f\in k|\
v_x(f)=0\}=\mathbb{C}$. Thus if $v\ne 0$, since
$\mathscr{T}X(X)\simeq k$, there exists a unique $h\in k^*$ such
that $v=hv_x$. So the constants of $v$ and the constants of $v_x$
coincide.

\begin{rem}
In broad terms what we do in this article is to study the
following geometric construction. Consider a matrix differential
equation
\[
v(f^i)=a^i_jf^j,\quad i\in\{1,\ldots,n\},\quad a^i_j\in k
\]
and an open $U\subset X$ over which we have a full-system of
solutions $(y^i_j)$, i.e. $y^i_j\in\mathscr{M}(U)$ for
$i,j\in\{1,\ldots,n\}$ and $\det(y^i_j)\ne 0$. The analytic map
\begin{eqnarray*}
U & \longrightarrow & GL_n(\mathbb{C})\\
p & \longmapsto       & (y^i_j(p))
\end{eqnarray*}
induces an algebraic map
\begin{eqnarray*}
U & \longrightarrow & GL_n(\mathbb{C})/G\\
p & \longmapsto       & (y^i_j(p))\cdot G
\end{eqnarray*}
where $G$ is the Galois group of our linear differential equation
which we will assume reductive.

Indeed, $\mathbb{C}[\textrm{GL}_n/G]\simeq \mathbb{C}[X^i_j]^G$
(Hilbert 14th), so let $P_1,\ldots,P_r$ be a set of generators of
$\mathbb{C}[X^i_j]^G$. In the coordinate system $(P_1,\ldots,P_r)$
of $\textrm{GL}_n(\mathbb{C})/G$, the map $p\mapsto
(y^i_j(p))\cdot G$ is given by $p\mapsto
(P_1(y^i_j(p)),\ldots,P_r(y^i_j(p)))$; and by Galois
Correspondence $P_l(y^i_j)=f_l\in k$. Therefore in the this
coordinate system, $p\mapsto (y^i_j(p))\cdot G$ is given by
$p\mapsto (f_1(p),\ldots,f_r(p))$.

Because the last map is algebraic, it can be extended to a
meromorphic map defined globally over $X$. The idea is to see to
which extend this last map characterizes our differential
equation.
\end{rem}

\begin{rem}
The maps above, $p\mapsto (y^i_j(p))$ and $p\mapsto
(y^i_j(p))\cdot G$, depend on the choices of a full-system of
solutions and of an open set $U$. Therefore, the first thing we
will do is to argue bi-rationally that the geometric properties of
the image of this maps does not depend on this choices (see Remark
\ref{fanowelldef}). For that we will need to provide a
coordinate-free description of this image, which will be done by
taking a symmetric algebra characterizing the image of the map
(Definition \ref{symalg} and Definition \ref{fanosymalg}).
\end{rem}

\subsection{Differential Modules}

\begin{rem}
We will use Einstein's notation for indices.
\end{rem}

Fix a non-trivial derivation $v\in\mathscr{T}X(X)$ of $k$. We
recall briefly the concept of differential module, which is a
coordinate-free description of a matrix differential equation. A
more detailed exposition may be found in~\cite{SvdP}.

\begin{defn}\label{defdiffmof}
A \emph{differential $k$-module} (rigorously a $(k,v)$-module) is
a finite dimensional $k$-vector space $M$ together with an
additive map
\[
\partial:M\rightarrow M
\]
satisfying the Leibnitz rule:
\begin{eqnarray*}
\partial(m_1+m_2) & = & \partial m_1+\partial m_2\qquad \forall (m_1,m_2)\in M^2\\
\partial fm & = & v(f)m+f\partial m\quad \forall (f,m)\in k\times M
\end{eqnarray*}
An $m\in M$ such that $\partial m=0$ is called a \emph{horizontal element}.
\end{defn}

\begin{rem}
If $f$ is a constant, i.e. if $v(f)=0$, the Leibnitz rule implies
that $\partial fm=f\partial m$. The collection of horizontal
elements thus forms a vector space over the field of constants
$\mathbb{C}$.
\end{rem}

\begin{rem}
Fix a basis $e_1,\ldots,e_n$ of $M$ and set
\[
\partial e_j=-a^i_je_i\quad \forall j\in\{1,\ldots,n\}.
\]
If $m=f^ie_i$, then
\begin{eqnarray*}
\partial m & = & v(f^i)e_i+f^i\partial e_i\\
           & = & v(f^i)e_i-f^ia^j_ie_j\\
           & = & (v(f^i)-a^i_jf^j)e_i.
\end{eqnarray*}
Solving the equation $\partial m=0$ therefore amounts to solve the
matrix differential equation
\[
v(f^i)=a^i_jf^j,\quad i\in\{1,\ldots,n\}.
\]
\end{rem}

\begin{prop}
Let $(M,\partial)$, $(M_1,\partial_1)$ and $(M_2,\partial_2)$ be three differential modules, then:
\begin{itemize}
\item[i)] The tensor product $M_1\otimes M_2$ inherits a
differential $k$-modules structure under the map:
\[
\partial_1\otimes\partial_2: m_1\otimes m_2\longrightarrow \partial_1 m_1\otimes m_2+ m_1\otimes \partial_2 m_2
\]
\item[ii)] The symmetric power $Sym^d(M)$ inherits a differential
$k$-module structure as a quotient of the tensor product of $d$
copies of $M$, $M^{\otimes d}$.
\item[iii)] The exterior power
$\bigwedge^d M$ inherits a differential $k$-module structure as a
quotient of $M^{\otimes d}$.
\item[iv)] The dual
$M^*=\textrm{Hom}_k(M,k)$ inherits a differential $k$-module
structure under the map:
\[
\partial^*: \mu\longmapsto [m\mapsto v(\mu(m))-\mu(\partial m)]
\]
\item[v)] The space of $k$-linear morphisms
$\textrm{Hom}_k(M_1,M_2)=M^*_1\otimes M_2$ inherits a differential
$k$-module structure $\partial_1^*\otimes\partial_2$. In
particular if $H\in\textrm{Hom}_k(M_1,M_2)$ is such that
$(\partial_1^*\otimes\partial_2) H=0$ then $\partial_2\circ
H=H\circ \partial_1$.
\end{itemize}
\end{prop}

\begin{proof}
Items $i)$, $ii)$, $iii)$, $iv)$ are in~\cite{SvdP}. To see $v)$,
write $H=f^i\otimes e_i$ for some $f^i\in M_1^*$ and some $e_i\in
M_2$, so that if $m\in M_1$ then $H(m)=f^i(m)e_i$. Now by
hypothesis $\partial_1^*f^i\otimes e_i+f^i\otimes \partial_2
e_i=0$, therefore
\begin{eqnarray*}
H(\partial_1 m) & = & f^i(\partial_1 m)e_i\\
                & = & v(f^i(m))e_i-\partial^*_1f^i(m)e_i\\
                & = & v(f^i(m))e_i+f^i(m)\partial_2 e_i\\
                & = & \partial_2 H(m)
\end{eqnarray*}
\end{proof}

\begin{rem}
Summarizing, a differential $k$-module endows a canonical
differential structure on any tensorial construction (duals,
tensor products, symmetric powers, exterior powers, sums,
$\ldots$) over $k$.
\end{rem}

\begin{rem}
Given a differential $k$-module $M$, we will not always be able to
find a basis composed of horizontal elements.
\end{rem}

The definition of a differential $k$-module depends on our choice
of a derivation on $k$ (cf. Definition \ref{defdiffmof}), in our
case it was $v$. Since we are in the realm of Riemann surfaces,
which are actually one dimensional manifolds over the complex
numbers, we can circumvent this restriction using connections.

\subsection{Connections and Pullbacks}

A meromorphic connection is a $\mathbb{C}$-linear map (linear over the constants)
\begin{eqnarray*}
\nabla: \mathscr{E} & \longrightarrow & \mathscr{E}\bigotimes_\mathscr{M}\Omega^1_\mathscr{M}
\end{eqnarray*}
satisfying the Leibnitz rule
\[
\nabla(fV)=V\otimes df+f\nabla V\qquad \forall (f,V)\in k\times\mathscr{E}(X)
\]
(recall $\mathscr{E}$ is the sheaf of meromorphic sections of the
vector bundle $\Pi: E\rightarrow X$, $\mathscr{M}$ the sheaf of
meromorphic functions and $\Omega^1_\mathscr{M}$ the sheaf of
meromorphic differential forms).

Given a meromorphic tangent field $v$, we define the $\mathscr{M}$-linear contraction map by
\begin{eqnarray*}
\imath_v:\ \mathscr{E}\bigotimes_\mathscr{M}\Omega^1_\mathscr{M} & \longrightarrow & \mathscr{E}\\
V\otimes\omega & \longmapsto & \omega(v)V.
\end{eqnarray*}
Not every element in $\mathscr{E}\otimes_k\Omega^1_\mathscr{M}$
can be written in the form $V\otimes\omega$, but as the map is
$\mathscr{M}$-linear, it suffices to define the map in such
elements. The composition of $\nabla$ followed by this contraction
map is denoted by $\nabla_v$. This is commonly called the
\emph{covariant derivative} along $v$.

\begin{rem}
The $k$-vector space of global sections $\mathscr{E}(X)$ of $\Pi$
is isomorphic to $k^n$, so $(\mathscr{E}(X),\nabla_v)$ is a
differential $k$-module.
\end{rem}

\begin{prop}
The connection $\nabla$ induces a $(k,v)$-module structure on
$\mathscr{E}(X)$ under the map $\nabla_v$. Conversely a
$(k,v)$-module structure on $\mathscr{E}(X)$ determines a
connection on $\mathscr{E}$.
\end{prop}

\begin{proof}
Derivations on $k$ are in a natural bijective correspondence with
global sections of $\mathscr{T}X$, and $\mathscr{T}X(X)\simeq k$.
Suppose $\nu\in\Omega^1_\mathscr{M}$ is the dual of a non-zero
derivation $v\in\mathscr{T}X$, i.e. $\nu(v)=1$. Because
$\Omega^1_\mathscr{M}$ is one-dimensional, for every
$\omega\in\Omega^1_\mathscr{M}$ we have $\omega=\omega(v)\nu$, and
we conclude that tensoring with $\nu$ is the inverse to
$\imath_v$. Define the following map $\nabla'$:
\begin{eqnarray*}
\nabla': \mathscr{E} &\longrightarrow & \mathscr{E}\bigotimes_\mathscr{M}\Omega^1_\mathscr{M}\\
   V & \longmapsto & \nabla_vV\otimes\nu.
\end{eqnarray*}
Then
\begin{eqnarray*}
\nabla'V & = & \nabla_vV\otimes\nu\\
           & = & \imath_v(\nabla V)\otimes\nu\\
           & = & \nabla V.
\end{eqnarray*}
Endowing $\mathscr{E}(X)$ with a differential $k$-module structure
is thereby seen to be the same as defining a meromorphic
connection over $\mathscr{E}$.
\end{proof}

\begin{rem}
As a corollary of the previous proposition we obtain that every
tensorial construction over meromorphic vector bundles with
connections inherits a connection in a similar way as it happens
with differential modules. This last remark is actually
independent of the fact that $X$ is one-dimensional.
\end{rem}

We now define connection pullbacks.

\begin{defn}
Take a meromorphic vector bundle $\Pi_0:E_0\rightarrow X_0$ over a
compact Riemann surface together with a meromorphic connection
$\nabla_0$ and a morphism $f:X\rightarrow X_0$. If
$\Pi:E\rightarrow X$ is the pullback bundle of
$\Pi_0:E_0\rightarrow X_0$, and
$(\overline{f},f):(E,X)\rightarrow(E_0,X_0)$ stands for the
canonical vector bundle morphism (i.e. $\Pi_0\circ
\overline{f}=f\circ\Pi$):
\[
\xymatrix{
 &  & E\ar@{-}[d]\ar[dll]_{\overline{f}}\\
E_0\ar@{-}[d] & & X\ar[dll]^{f} \\
X_0 & &
}
\]
the \emph{pullback} (\emph{connection}) $f^*\nabla_0$ at $p\in X$ is given by
\[
[(f^*\nabla_0)_v V] (p)= [(\nabla_0)_{f_*v}\overline{f} V] (f(p))
\]
(the vector field $f_*v$ is well defined only locally, but since
$(\nabla_0)_\centerdot\bullet$ is tensorial on $\centerdot$ the
value of $[(\nabla_0)_{f_*v}\overline{f} V] (f(p))$ is uniquely
determined by $f_*v(p)\ $).
\end{defn}

\begin{rem}
In terms of sections of vector bundles, it follows from the
definition that $\nabla$ is equal to $f^*\nabla_0$ if and only if
the section $V$ of $\Pi$ is horizontal (i.e. $\nabla V=0$) is
equivalent to $\overline{f} V$, as a section of $\Pi_0$, is
horizontal.
\end{rem}

\subsection{Symmetric algebra of first integrals and Fano curve}

We now study symmetric products and duals of vector bundles with
connections in more detail. For the remainder of this section we
set $M=\mathscr{E}(X)$ and we fix a non-zero derivation
$v\in\mathscr{T}X(X)$.

\begin{defn}\label{deflfi}
Let $\phi\in M^*=\textrm{Hom}_k(\mathscr{E}(X),k)$. We say that
$\phi$ is a \emph{linear first integral} (of $(E,\nabla)$) if
$\phi(V)\in k$ is constant whenever $V$ is horizontal.
\end{defn}

\begin{rem}
Let $\phi$ be a linear first integral and let $V$ be horizontal.
Then
\begin{eqnarray*}
0 & = & v(\phi(V))\\
  & = & [\nabla^*_v\phi](V).
\end{eqnarray*}
Therefore, by taking a full system of solutions we see that
$\nabla^*_v\phi=0$ if and only if $\phi$ is a linear first
integral.  In particular, linear first integrals form a vector
space over the constants.
\end{rem}

\begin{defn}\label{symalg}
Denote by $\textrm{S}^d_k(M)$ the $\mathbb{C}$-vector space of
linear first integrals of the $d$-th symmetric product of
$(E,\nabla)$. As a convention we set
$\textrm{S}^0_k(M)=\mathbb{C}$. We define the graded
$\mathbb{C}$-algebra of linear first integrals of $M$ as
\[
\textrm{S}_k(M)=\bigoplus_{d\ge 0}\textrm{S}^d_k(M).
\]
\end{defn}

\begin{rem}
If $V\in M$ is horizontal, then so is $V^d$ in the $d$-th
symmetric power. Thus, given $\phi\in\textrm{S}^d_k(M)$,
$\phi(V^d)$ is constant. The elements in $\textrm{S}^d_k(M)$ are
called \emph{($d$-th order) first integrals} of $M$. Once we fix a
basis of $M^*$, i.e. a coordinate system for $M$, an element in
$\textrm{S}^d_k(M)$ is given by a homogeneous polynomial of order
$d$ in the coordinates on $M$. With this in mind $\textrm{S}_k(M)$
corresponds to the collection of rational $\mathbb{C}$-valued
functions over $E$ which contain horizontal sections of
$\mathscr{E}$ within their level sets.
\end{rem}

\begin{rem}
If we pick a $V\in M$ we obtain a homomorphism
\[
V: \textrm{S}_k(M)\longrightarrow k
\]
of $\mathbb{C}$-algebras by evaluating each first integral in
$\textrm{S}^d_k(M)$ at $V^d$.
\end{rem}

\begin{rem}\label{evV}
Let $U\subseteq X$ be an open set and pick $V\in\mathscr{E}(U)$,
then we can evaluate first integrals in $\textrm{S}^d_k(M)$ at $V$
to obtain an element of $\mathscr{M}(U)$. Indeed, since the first
integrals in $\textrm{S}^d_k(M)$ are globally defined meromorphic
functions, we can restrict them to $\Pi^{-1}U$ and evaluate them at
$V^d$. So in this case $V$ can be identified with a homomorphism
\[
V: \textrm{S}_k(M)\longrightarrow \mathscr{M}(U)
\]
of $\mathbb{C}$-algebras. Note that the previous remark is the particular case $U=X$.
\end{rem}

\begin{rem}\label{ftode}
The fundamental theorem of ordinary differential equations
guarantees that if $p$ is not a singular point of $\nabla$, then
for a sufficiently small open set $U\subseteq X$ containing $p$ we
can find a frame $(V_1,\ldots, V_n)$ of $\mathscr{E}(U)$ composed
of horizontal elements.
\end{rem}

\begin{defn}\label{fanosymalg}
Let $H$ be an invertible element of the differential $k$-module
\[
\textrm{Hom}_k(\mathscr{E}(U),\mathscr{E}(U))\simeq [\mathscr{E}^*\otimes_\mathscr{M}\mathscr{E}](U)
 \]
associating to a global frame of $M=\mathscr{E}(X)$, restricted to
$U$, a frame of $\mathscr{E}(U)$ composed of horizontal sections
(fixing a basis for $M^*$, $H$ corresponds to the matrix of
coordinates of a full system of solutions). The \emph{Fano curve}
of $(E,\nabla)$ is defined as the $\mathscr{M}(U)$-valued point
\[
H: \textrm{S}_k(\textrm{Hom}_k(\mathscr{E}(X),\mathscr{E}(X))) \longrightarrow \mathscr{M}(U)
\]
(see Remark \ref{evV}). If
$\textrm{S}_k(\textrm{Hom}_k(\mathscr{E}(X),\mathscr{E}(X)))$ is
finitely generated we define the \emph{projective Fano curve}
$X_0$ to be the non-singular model of the projective variety
defined by the maximal homogeneous ideal contained in the kernel
of the Fano curve (i.e. the maximal homogeneous ideal contained in
the kernel of $H$).
\end{defn}

\begin{rem}
To obtain the polynomials defining the projective Fano curve we
rely on the algorithm by M. van Hoeij and J.-A. Weil ~\cite{We}.
Using their terminology these polynomials correspond to the
homogeneous ``invariants'' with vanishing ``dual first
integrals''.
\end{rem}

\begin{rem}\label{fanowelldef}
There are many aspects of this definition that require elaboration.
\begin{itemize}
\item[$\bullet$] The Fano curve is $k$-valued: this is a
consequence of the Galois correspondence. This fact, together with
the ideal defining the Fano curve will be studied in the next
section (Proposition \ref{fanoideal}).
\item[$\bullet$] The Fano
curve is independent of the choice of $H$ up to isomorphism: if
$\tilde{H}$ is another invertible element of
$\textrm{Hom}_k(\mathscr{E}(U),\mathscr{E}(U))$ associating to a
global frame of $M=\mathscr{E}(X)$ a frame of $\mathscr{E}(U)$
composed of horizontal sections, then it differs from $H$ by an
element of $GL_n(\mathbb{C})$ multiplying on the right and an
element of $GL_n(k)$ multiplying on the left, which one can see
them as acting on
$\textrm{S}_k(\textrm{Hom}_k(\mathscr{E}(U),\mathscr{E}(U)))$
sending one Fano curve to another.
\item[$\bullet$] The Fano curve
is also independent of the choice of $U$ up to isomorphism: let
$\tilde{U}$ be another open set where one can define a frame
composed of horizontal sections, then by taking a path from $U$ to
$\tilde{U}$ we can prolong the frame holomorphically to a frame
over $\tilde{U}$. Because the first integrals are constant over
horizontal sections and are globally defined, the prolongation
does not change the $k$-valued point $H$ (cf. first $\bullet$ in
this Remark).
\end{itemize}
\end{rem}

\begin{rem}
Let us see how we can obtain geometrically the projective Fano
curve under the assumption that the Galois group $G$ of the
equation is reductive. Fixing a basis for $M$ let $H$ be
represented over the open set $U\subseteq X$ by the full-system of
solutions $(y^i_j)$, i.e. $y^i_j\in\mathscr{M}(U)$ for
$i,j\in\{1,\ldots,n\}$. The image of the analytic map
\begin{eqnarray*}
U & \longrightarrow & GL_n(\mathbb{C})\\
p & \longmapsto       & (y^i_j(p))
\end{eqnarray*}
corresponds to a solution curve in the phase portrait. Composing
this map with the canonical projection
$GL_n(\mathbb{C})\rightarrow GL_n(\mathbb{C})/G$, we obtain an
algebraic map
\begin{eqnarray*}
U & \longrightarrow & GL_n(\mathbb{C})/G\\
p & \longmapsto       & (y^i_j(p))\cdot G
\end{eqnarray*}
Therefore the Zariski closure of the image is an algebraic curve.
The projective Fano curve corresponds to the projective variety
defined by the cone over this algebraic curve. We will establish
the invariance of the Fano curve under ``projective equivalence''
(Proposition \ref{propalg1}).
\end{rem}

\subsection{Projective equivalence}

Assume now that $P:L\rightarrow X$ is a $1$-dimensional vector
bundle (a line bundle), and fix a global non-zero meromorphic
section $s\in\mathscr{L}(X)$. We can then identify $\mathscr{E}$
with $\mathscr{L}\otimes\mathscr{E}$ through the morphism
$V\mapsto s\otimes V$. It must be noted that this identification
is not unique, since it depends on the choice of $s$.

\begin{defn}\label{defnprojequiv}
Given another meromorphic connection $\nabla'$ on
$\Pi:E\rightarrow X$, we say that $\nabla$ and $\nabla'$ are
\emph{projectively equivalent} if there exist a $1$-dimensional
meromorphic vector bundle $P:L\rightarrow X$ with a connection
$\nabla_1$ such that
\[
\nabla'\simeq\nabla_1\otimes\nabla.
\]
\end{defn}

\begin{rem}
Assume that $s\otimes\cdot: V\mapsto s\otimes V$ is a horizontal
morphism from $(\mathscr{E},\nabla')$ to
$(\mathscr{L}\otimes\mathscr{E},\nabla_1\otimes\nabla)$, then
\[
s\otimes \nabla'(V)=\nabla_1 s\otimes V+s\otimes\nabla V.
\]
\end{rem}

\begin{prop}
Projective equivalence is an equivalence relation on the
collection of connections over $\mathscr{E}$.
\end{prop}

\begin{proof}
Reflexivity and Transitivity is immediate. We prove symmetry. Note that since $(P,\nabla_1): L\rightarrow X$ is $1$-dimensional,
the same holds for the dual $(L^*,P^*,\nabla_1^*)$. Under the
canonical isomorphisms we have
\[
\mathscr{L}^*\otimes\mathscr{L}\simeq \textrm{Hom}_{\mathscr{M}}(\mathscr{L},\mathscr{L})\simeq \mathscr{M}.
\]
Moreover in terms of the basis $s=s_1$ of $\mathscr{L}(X)$, if
$s^1\in\mathscr{L}^*$ is such that $s^1(s_1)=1$, then
\begin{eqnarray*}
\{[\nabla_1^*\otimes\nabla_1]_v(s^1\otimes s_1)\}(s_1) & = & \{\nabla_{1v}^*s^1\otimes s_1+s^1\otimes \nabla_{1v} s_1\}(s_1) \\
 & = & \{\nabla_{1v}^*s^1\}(s_1)s_1+s^1(s_1)\nabla_{1v}s_1\\
 & = & \{v(s^1(s_1))-s^1(\nabla_{1v}s_1)\}s_1+\nabla_{1v}s_1\\
 & = & \nabla_{1v}s_1-s^1(\nabla_{1v}s_1)s_1 = 0.
\end{eqnarray*}
So the connection on $\mathscr{L}^*\otimes\mathscr{L}$ is trivial,
and if $\nabla'\simeq\nabla_1\otimes\nabla$ then
\[
\nabla_1^*\otimes\nabla' \simeq \nabla_1^*\otimes\nabla_1 \otimes \nabla\simeq\nabla.
\]
We conclude that projective equivalence is an equivalence
relation.
\end{proof}

\subsection{The geometric Galois group}\label{galgr}

\begin{rem}
Let $\phi\otimes V\in
[\mathscr{E}^*\otimes_\mathscr{M}\mathscr{E}](X)$, so that:
\begin{eqnarray*}
[\nabla^*\otimes\nabla](\phi\otimes V) & = & \nabla^*\phi\otimes V+\phi\otimes\nabla V.
\end{eqnarray*}
Then under the canonical isomorphism
$\textrm{Hom}_k(\mathscr{E}(X),\mathscr{E}(X))\simeq
[\mathscr{E}^*\otimes_\mathscr{M}\mathscr{E}](X)$ we obtain:
\begin{eqnarray*}
[\nabla^*\otimes\nabla]_v(\phi\otimes V)\ (W)& = & [\nabla^*_v\phi](W) V+\phi(W)\nabla_v V\\
  & = & \{v(\phi(W))-\phi(\nabla_v W)\}V+\phi(W)\nabla_v V\\
  & = & \nabla_v[\phi(W)V]-\phi(\nabla_v W)V\\
  & = & \nabla_v[(\phi\otimes V)\ (W)]-(\phi\otimes V)\ (\nabla_v W).
\end{eqnarray*}
This implies that $\psi\in
\textrm{Hom}_k(\mathscr{E}(X),\mathscr{E}(X))$ is horizontal if
and only if it is a connection preserving map, i.e.
\[
\nabla[\psi(W)]=\psi(\nabla W)
\]
(in terms of differential $k$-modules, $H$ is horizontal if and only if $H$ is a morphism of differential modules).
\end{rem}

Fix a horizontal automorphism $\psi\in
\textrm{Hom}_k(\mathscr{E}(X),\mathscr{E}(X))$ (e.g. $\psi$ is the
identity). The collection of horizontal automorphisms $\phi\in
\textrm{Hom}_k(\mathscr{E}(X),\mathscr{E}(X))$ defining the same
map
\[
\phi: \textrm{S}_k(\textrm{Hom}_k(\mathscr{E}(X),\mathscr{E}(X))) \longrightarrow \mathbb{C}
\]
as $\psi$ will be denoted by $[\psi]$. In particular, the elements
in $\psi^{-1}[\psi]$ form a group, which we will call the
\emph{geometric Galois group} of $\nabla$. Note that $\phi$ is
$\mathbb{C}$-valued because first integrals are constant on
horizontal elements.

\begin{rem}\label{remgeogal}
As in the case of the Fano curve, it follows that the geometric
Galois group is independent of $\psi$ and $U$, up to isomorphism.
More generally, we could replace $\psi$ by any (non-necessarily
horizontal) automorphism in
$\textrm{Hom}_k(\mathscr{E}(U),\mathscr{E}(U))$. Indeed, in such
case we would obtain a conjugate of $G$ in $GL_n(\mathscr{M}(U))$.
\end{rem}

\begin{rem}
The geometric Galois group measures the horizontal automorphisms
that cannot be distinguished one from another by means of first
integrals.
\end{rem}

\section{Algebraic Interpretation}

To give an algebraic interpretation of the geometric objects
introduced above, we require a global meromorphic frame
$F=(e_1,\ldots,e_n)\in\mathscr{E}(X)^n$ of $E\rightarrow X$ so
that we can do some computations using coordinates.

As above, we fix a derivation $v\in \mathscr{T}X(X)$. We discussed
earlier that the solutions to the equation $\nabla_vV=0$, are
given by the solutions to the matrix differential equation
\begin{eqnarray}\label{eq1}
v(f^i)=a^i_jf^j\quad i\in\{1,\ldots,n\},
\end{eqnarray}
where $f^i=e^i(V)$ and $\nabla_ve_i=-a^j_ie_j$.

\subsection{Picard-Vessiot Extensions and Galois groups}

Let $H=(V_1,\ldots,V_n)$ be a frame of $\mathscr{E}(U)$ composed
of horizontal elements. If we denote by
$y^i_j=e^i(V_j)\in\mathscr{M}(U)$ the coordinates of these
horizontal elements in our original frame $F$, then
\[
v(y^i_j)=a^i_ky^k_j
\]
and the Picard-Vessiot extension is given by the subfield
\[
K:=k(y^i_j)
\]
of $\mathscr{M}(U)$. The inclusion map into this extension is
given by the restriction map:
\begin{eqnarray*}
k\simeq\mathscr{M}(X) & \longrightarrow & \mathscr{M}(U) \\
                          f & \longmapsto     & f\!\!\upharpoonright_U.
\end{eqnarray*}
Formally, the Picard-Vessiot extensions can be obtained as follows
(a rigorous exposition may be found in \cite{SvdP}). Consider the
ring of polynomials in $n\times n$ variables with coefficients in
$k$,
\[
k[X^i_j,\frac{1}{\det}]_{i,j\in\{1,\ldots,n\}},
\]
inverting the determinant polynomial $\det:=\det(X^i_j)$. We turn
this ring into a differential ring extension of $(k,v)$ by setting
\[ v(X^i_j)=a^i_lX^l_j
\]
and using the Leibniz rule and the quotient rule
\[
v(ab^{-1})=[v(a)b-av(b)]b^{-2}
\]
we extend the derivation to the whole ring. An ideal $I\subseteq
k[X^i_j,\frac{1}{\det}]$ is differential if it is closed under
derivation, i.e. $v(I)\subset I$. Maximal differential ideals are
prime. We obtain a Picard-Vessiot extension for $\nabla$ by taking
the fraction field of the quotient of $k[X^i_j,\frac{1}{\det}]$ by
a maximal differential ideal $I$.

Note that we can make $\textrm{GL}_n(\mathbb{C})$ act on
$k[X^i_j,\frac{1}{det}]$ by differential automorphisms over $k$ by
setting for $(g^i_j)\in \textrm{GL}_n(\mathbb{C})$
\begin{eqnarray*}
(g^i_j): k[X^i_j,\frac{1}{\det}] & \longrightarrow & k[X^i_j,\frac{1}{\det}]\\
X^i_j   & \longmapsto & X^i_lg^l_j
\end{eqnarray*}
We can identify the Galois group $G$ with the elements of
$\textrm{GL}_n(\mathbb{C})$ sending $I$ to itself. In particular,
we may take $I$ as the kernel of the evaluation map of
$k$-algebras:
\begin{eqnarray*}
\Psi: k[X^i_j,\frac{1}{\det}] & \longrightarrow & K\subseteq\mathscr{M}(U) \\
            X^i_j      & \longmapsto     & y^i_j
\end{eqnarray*}

Let us consider with more care the relationship between the
geometric Galois group and the (algebraic) Galois group. As above,
$H$ will denote the invertible element $e^i\otimes
V_i\in\textrm{Hom}_k(\mathscr{E}(U),\mathscr{E}(U))$. An element
of $S^d_k(\textrm{Hom}_k(\mathscr{E}(X),\mathscr{E}(X)))$
corresponds on the frame $F$ to a homogeneous polynomial
$P(X^i_j)$ of degree $d$ such that $P(y^i_j)=f$, where $f\in k$.
So $P(X^i_j)-f$ is in the kernel $I$ of the evaluation map $\Psi$.
Moreover, if $g=(g^i_j)\in \textrm{GL}_n(\mathbb{C})$ is in our
algebraic definition of the Galois group, then $P(X^i_lg^l_j)-f$
is again in $I$, meaning that $P(y^i_lg^l_j)=f$, so the geometric
Galois group contains the (algebraic) Galois group (cf. Remark
\ref{remgeogal}).

\begin{rem}\label{thmcompoint2}
In \cite[Theorem 4.2]{compoint} Compoint proves that there is a bijective correspondence between first integrals in $S^d_k(\textrm{Hom}_k(\mathscr{E}(X),\mathscr{E}(X)))$ and homogeneous elements of $\mathbb{C}[X^i_j]^G$ of degree $d$. Furthermore, the theorem also states (see Theorem \ref{compthm} below) that when the (algebraic) Galois group $G$ is reductive and unimodular, the elements of $I$ of the form $P(X^i_j)-f$, where $P(X^i_j)\in\mathbb{C}[X^i_j]^G$ and $P(y^i_j)=f$, generate $I$, so in such case the geometric and the (algebraic) Galois group coincide. In view of Compoint's result we will restrict ourselves to such case.
\end{rem}

\subsection{Compoint's Theorem and the projective Fano Curve}

In analogy with classical Galois theory, differential Galois
theory also admits a Galois correspondence. In particular, the
fixed field $K^G$ is the ground field $k$. Furthermore using the
action in the previous section, if $P(X^i_j)\in k[X^i_j]$ is
invariant under the action of $G$, then the Galois correspondence
implies $P(y^i_j)\in k$.

\begin{rem}
For the rest of this section we will assume that $G$ is unimodular
and reductive. In particular this will imply that
$\textrm{S}_k(\textrm{Hom}_k(\mathscr{E}(X),\mathscr{E}(X)))$ is
finitely generated, and so we will be able to associate to
$\nabla$ a projective Fano curve $X_0$. We now introduce a theorem
that allows us to effectively compute $X_0$.
\end{rem}

\begin{thm}[Compoint \cite{compoint}]\label{compthm}
If $G$ is reductive and unimodular, then $I$ is generated by the
$G$-invariants it contains. Moreover, if $P_0,\ldots, P_r$ is a
set of generators for the $\mathbb{C}$-algebra of $G$-invariants
in $\mathbb{C}[X^i_j]$, and if $f_0,\ldots, f_r\in k$ are such
that $P_i-f_i\in I$, then $I$ is generated over
$k[X^i_j,\frac{1}{det}]$ by $P_i-f_i$, $i\in\{0,\ldots, r\}$.
\end{thm}

\begin{rem}\label{rem1}
Compoint's theorem says that $I$ is uniquely determined by the
restriction of $\Psi: k[X^i_j,\frac{1}{det}]\rightarrow K$ to
$\mathbb{C}[X^i_j]^G\rightarrow k$.
\[
\xymatrix{
 &  & K\ar@{-}[d]\ar@{-}[dll]\\
\mathbb{C}[X^i_j]\ar@{-}[d] & & k\ar@{-}[dll] \\
\mathbb{C}[X^i_j]^G & &
}
\]
\end{rem}

\begin{prop}\label{fanoideal}
We keep the notation and hypotheses of the theorem and the remark.
Let $J$ be the maximal homogeneous ideal contained in the kernel
of $\Psi\upharpoonright_{\mathbb{C}[X^i_j]^G}$. The projective
variety $Z(J)\subseteq\textrm{Proj}\ (\mathbb{C}[X^i_j]^G)$
defined by the homogeneous ideal $J$, is bi-rationally equivalent to the
projective Fano curve $X_0$ of $\nabla$. Moreover, we have
\[
\mathbb{C}(X_0)=\mathbb{C}(\frac{f_i^{m_i}}{f_j^{m_j}}|m_in_i=m_jn_j,\ f_j\ne 0).
\]
\end{prop}

\begin{proof} Let $P(X^i_j)\in J$ of degree $d$, then $P(X^i_j)$ corresponds, in
our frame $F$,  to a first integral of degree $d$ (see Remark \ref{thmcompoint2}), i.e. a homogeneous element
$\lambda\in\textrm{S}_k(\textrm{Hom}_k(\mathscr{E}(X),\mathscr{E}(X)))$,
vanishing at $H=(V_1,\ldots,V_n)=e^i\otimes V_i$ (being pedantic,
it vanishes at $H^d$). This means that $\lambda$ is in the
homogeneous ideal defining the projective Fano curve. Conversely
\cite[Theorem 4.2]{compoint} says that such a $\lambda$, i.e. a
first integral of degree $d$ vanishing at $H$, corresponds to a
homogeneous polynomial invariant under the $G$-action of degree
$d$, $P(X^i_j)\in \mathbb{C}[X^i_j]$, such that $P(y^i_j)=0$, i.e.
$P(X^i_j)\in J$.

Fix $P_l\in\{P_0,\ldots,P_r\}$ such that
$f_l=\psi(P_l)=P_l(y^i_j)\ne 0$, let $U_{P_l}=\{P_l\ne
0\}\subseteq\textrm{Proj}\ (\mathbb{C}[X^i_j]^G)$. The ring of
coordinate functions of $U_{P_l}$ is
\[
\mathbb{C}[U_{P_l}]=\mathbb{C}[\frac{P_i^{m_i}}{P_l^{m_l}}|m_in_i=m_ln_l],
\]
So
\[
\mathbb{C}[U_{P_l}\cap X_0]=\mathbb{C}[\frac{f_i^{m_i}}{f_l^{m_l}}|m_in_i=m_ln_l],
\]
The statement follows after taking the quotient field.\qed
\end{proof}

\begin{lem}\label{lem0}\label{covgr}
We keep the notation and hypotheses of Compoint's theorem. Let
$\textrm{pr}: SL_n(\mathbb{C})\rightarrow PSL_n(\mathbb{C})$ and
set
\[
\widetilde{G}:=\textrm{pr}^{-1}(PG),
\]
where $PG=G/\mu_n$ ($\mu_n=Z(SL_n(\mathbb{C}))$). If $P(y^i_j)=1$
for some homogeneous $P(X^i_j)\in\mathbb{C}[X^i_j]^G$ of degree
$n$, then $\mathbb{C}(X_0)=\mathbb{C}(y^i_j)^{\widetilde{G}}$.
\end{lem}

\begin{proof}
As in the latter proposition we take $U_P=\{P\ne 0\}$, so
$X_0\subseteq U_P$ and
$\mathbb{C}[X_0]=\mathbb{C}[f_i^{m_i}|m_in_i=lcm(m_i,n)]$. Recall
$f_i^{m_i}=P_i^{m_i}(y^i_j)$. Now because
$n|\textrm{deg}(P_i^{m_i})$ then
$P_i^{m_i}\in\mathbb{C}[X^i_j]^{\mu_n}$, but
$P_i^{m_i}\mathbb{C}[X^i_j]^G$ so
$P_i^{m_i}\in\mathbb{C}[X^i_j]^{\widetilde{G}}$, whence
$\mathbb{C}[X_0]\subseteq \mathbb{C}[y^i_j]^{\widetilde{G}}$.
Conversely
$\mathbb{C}[X^i_j]^{\widetilde{G}}\subseteq\mathbb{C}[X^i_j]^G$,
so
$\mathbb{C}[y^i_j]^{\widetilde{G}}\subseteq\mathbb{C}[X_0]$.\qed
\end{proof}

\begin{rem}
A particular case of the previous lemma occurs when
\[
\det(y^i_j)=1.
\]
\end{rem}

\begin{lem}\label{lem1}
Under the hypotheses of Compoint's theorem, if $\det(y^i_j)=1$
there is a choice of $v\in\mathscr{T}X$ such that in (\ref{eq1})
one has
\[
(a^i_j)\in M_{n\times n}(\mathbb{C}(X_0))+\mathfrak{g}(k),
\]
where $\mathfrak{g}$ is the Lie algebra of $G$.
\end{lem}

\begin{proof} Let $P_l=P_l(X^i_j)\in\mathbb{C}[X^i_j]^{\widetilde{G}}$,
for $l\in\{1,\ldots,r\}$ be generators of the
$\widetilde{G}$-invariant subalgebra
$\mathbb{C}[X^i_j]^{\widetilde{G}}$. Let $f_l\in k$ be such that
$P_l-f_l\in I$, $l\in\{1,\ldots,r\}$. We denote by $\frac{\partial
P_l}{\partial X^\iota_\kappa}(X^i_j)$ the partial derivative of
$P_l$ with respect to $X^\iota_\kappa$.

We introduce a differential field
$\mathbb{C}(y^i_j)(b^\iota_\kappa)_{\iota,\kappa\in\{1,\ldots,n\}}$,
where the $b^\iota_\kappa$'s are variables, with derivation
$\tilde{v}_0$ defined by:
\begin{eqnarray*}
\tilde{v}_0(y^i_j) & = & b^i_ky^k_j\\
\tilde{v}_0(b^i_j) & = & 0.
\end{eqnarray*}
Using the chain rule we obtain
\begin{eqnarray*}
\tilde{v}_0(P_l(y^i_j)) & = & \frac{\partial P_l}{\partial X^\iota_\kappa}(y^i_j)\tilde{v}_0(y^\iota_\kappa)\\
                        & = & \frac{\partial P_l}{\partial X^\iota_\kappa}(y^i_j)b^\iota_\lambda y^\lambda_\kappa.
\end{eqnarray*}
We extend the action of $\widetilde{G}$ on $\mathbb{C}(y^i_j)$ to
$\mathbb{C}(y^i_j)(b^\iota_\kappa)_{\iota,\kappa\in\{1,\ldots,n\}}$
by letting each $b^\iota_\kappa$ be fixed by $\widetilde{G}$. The
chain rule then implies that for $(g^i_j)\in \widetilde{G}$ one
has
\begin{eqnarray*}
\tilde{v}_0(P_l(y^i_lg^l_j)) & = & \frac{\partial P_l}{\partial X^\iota_\kappa}(y^i_lg^l_j)\tilde{v}_0(y^\iota_\mu g^\mu_\kappa)\\
                        & = & \frac{\partial P_l}{\partial X^\iota_\kappa}(y^i_lg^l_j)b^\iota_\lambda y^\lambda_\mu g^\mu_\kappa.
\end{eqnarray*}
The equality $P_l(y^i_j)=P_l(y^i_lg^l_j)$ in turn implies
\[
\frac{\partial P_l}{\partial X^\iota_\kappa}(y^i_j)y^\lambda_\kappa =
\frac{\partial P_l}{\partial X^\iota_\kappa}(y^i_lg^l_j)y^\lambda_\mu g^\mu_\kappa\quad\forall\ \iota,\lambda,
\]
for all $(g^i_j)\in \widetilde{G}$, whence $\frac{\partial
P_l}{\partial X^\iota_\kappa}(y^i_j)y^\lambda_\kappa\in
\mathbb{C}(y^i_j)^{\widetilde{G}}$ for each $\iota,\lambda$.

Let $v$ be a non-trivial derivation of
$\mathbb{C}(X_0)=\mathbb{C}(y^i_j)^{\widetilde{G}}$ (unless
$\mathbb{C}(X_0)=\mathbb{C}$ in which case we let $v$ be any
element in $\mathscr{T}X$). Note that since $P_l$ is
$\widetilde{G}$-invariant we have
$P_l(y^i_j)=f_l\in\mathbb{C}(X_0)\subseteq \mathbb{C}(y^i_j)$.
Consider the following system of linear equations in the variables
$b^\iota_\lambda$ with coefficients in $\mathbb{C}(X_0)$:
\begin{eqnarray}\label{syseq}
\frac{\partial P_l}{\partial X^\iota_\kappa}(y^i_j) y^\lambda_\kappa b^\iota_\lambda=v(f_l)\quad l\in\{0,\ldots, r\}.
\end{eqnarray}
This system has solutions in $k$ (apply $v$ on both sides of the equalities $P_l(y^i_j)=f_l$ in $K$). The system of equations is therefore consistent, and the system can thus be solved in the field of coefficients $\mathbb{C}(X_0)$. Specialize $b^\iota_\lambda$ to such solutions, so that $(b^i_j)\in M_{n\times n}(\mathbb{C}(X_0))$. When we apply $v$ on $P_l(y^i_j)=f_l$ in $K$, we obtain the solutions $a^\iota_\lambda$ to (\ref{syseq}). Hence the $(a^i_j)-(b^i_j)$ is a solution to the homogeneous system associated to (\ref{syseq}); but the left hand side of the equations in the system are the polynomials defining $\mathfrak{g}$, so $(a^i_j)-(b^i_j)\in\mathfrak{g}(k)$.\qed\end{proof}

\begin{rem} We have
\[
\xymatrix{
\mathbb{C}(y^i_j)\ar@{.>}[rr] &  & K\\
\mathbb{C}(y^i_j)^{Z(\widetilde{G})}\ar@{-}[d]\ar@{^{(}->}[urr]\ar@{.}[u] & & k\ar@{-}[u] \\
\mathbb{C}(X_0)\ar@{^{(}->}[urr] & &
}
\]
\end{rem}

\begin{defn}
We say that the projective Fano curve $X_0$ is \emph{degenerate}
if it is is not $1$-dimensional, i.e. if
$\mathbb{C}(X_0)=\mathbb{C}$.
\end{defn}

\begin{prop}\label{propfd}
The projective Fano curve is degenerate if and only if $G$ is
connected and $\mathbb{C}(y^i_j)$ corresponds to the field of
rational functions over a coset of $G$ in
$\textrm{GL}_n(\mathbb{C})$.
\end{prop}

\begin{proof} Let $I_\mathbb{C}$ be the kernel of the evaluation map
$\Psi\upharpoonright:\mathbb{C}[X^i_j,\frac{1}{\det}]\rightarrow
K$ sending $X^i_j\mapsto y^i_j$. Then $I_\mathbb{C}$ is a prime
ideal. $I_\mathbb{C}$ is invariant under the $G$-action, so
passing to the quotient we see that $(I_\mathbb{C})^G=I_G$ is the
kernel of the restriction of the evaluation map to
$\mathbb{C}[X^i_j,\frac{1}{\det}]^G\rightarrow K^G=k$. Again,
$I_G$ is a prime ideal. Because the Fano curve is degenerate, the
maximal homogeneous ideal $J$ contained in $I_G$ corresponds at
the level of varieties to a line, and $I_G$ to a point in that
line. So we conclude that $G$ acts transitively by left
multiplication over the subvariety of $\textrm{GL}_n(\mathbb{C})$
defined by $I_\mathbb{C}$. In other words, $I_\mathbb{C}$ defines
a coset of $G$. Finally, because $I_\mathbb{C}$ is prime, this
coset is connected, whence $G$ is connected. This proves the
necessity in the statement of the proposition. The sufficiency
follows immediately by noting that $G$ acts transitively on its
cosets.\qed\end{proof}

\begin{rem}
When $G$ is connected the field of rational functions over a coset
of $G$ is isomorphic to $\mathbb{C}(G)$ because the coset and $G$
are isomorphic varieties.
\end{rem}

\begin{rem}
When $\mathbb{C}(X_0)=\mathbb{C}$, the system (\ref{syseq}) is
homogeneous so $(a^i_j)\in\mathfrak{g}(k)$.
\end{rem}

\subsection{Projective Equivalence and Pullbacks}

Let us consider the algebraic properties of projective equivalence and of pullbacks.

Let us start with projective equivalences. We put another
meromorphic connection $\nabla_1$ on $\Pi:E\rightarrow X$ and we
assume that $\nabla$ and $\nabla_1$ are projectively equivalent.
This means that there is a $1$-dimensional meromophic bundle
$P:L\rightarrow X$ with connection $\nabla'$ such that $\nabla_1$
can be identified with $\nabla'\otimes\nabla$. We make this
identification explicit by fixing (as before) a global frame
$(e_1,\ldots,e_n)$ for $\mathscr{E}(X)$ and a non-zero global
section $s^1\in\mathscr{L}^*(X)$ such that the mapping $V\mapsto
s_1\otimes V$ is a horizontal isomorphism
$(E,\Pi,\nabla_1)\rightarrow (L\otimes_X
E,P\otimes\Pi,\nabla'\otimes\nabla)$, where $s^1(s_1)=1$.

Let $U\subseteq X$ be an open set avoiding the singularities of
$\nabla$ and of $\nabla'$, and let $h\in\mathscr{L}(U)$ be such
that $\nabla'h=0$. We set $f:=s^1(h)$. As usual, we choose
$(V_1,\ldots,V_n)$ a local horizontal frame of $\mathscr{E}(U)$
with respect to $\nabla$, and $y^i_j=e^i(V_j)$. In particular the
coordinates of $h\otimes V_j$ on the frame $(s_1\otimes
e_1,\ldots,s_1\otimes e_n)$ are $s^1\otimes e^i(h\otimes
V_j)=fy^i_j$.

Let $W_j\in\mathscr{E}(U)$ be defined for $j\in\{1,\ldots,n\}$ by
\[
W_j=fy^i_je_i,
\]
so that $s^1\otimes e^i(s_1\otimes W_j)=fy^i_j$ and $s_1\otimes W_j=h\otimes V_j$. Thus
\[
\nabla_1W_j=\nabla'\otimes\nabla(s_1\otimes W_j)=\nabla'\otimes\nabla(h\otimes V_j)=0.
\]
In particular, a Picard-Vessiot extension for $\nabla_1$ is generated by $(fy^i_j)$.

\begin{prop}\label{propalg1}
Under the hypotheses of Compoint's theorem, if $\nabla$ and $\nabla_1$ are projectively equivalent then their projective Fano curves coincide.
\end{prop}

\begin{proof} Let $P(X^i_j)\in\mathbb{C}[X^i_j]^G$ be homogeneous of degree $d$ such that $P(y^i_j)=0$. Then $P(fy^i_j)=f^dP(y^i_j)=0$.\qed\end{proof}

Now let us turn our attention to pullbacks. Let
$\Pi_0:E_0\rightarrow X_0$ be an $n$-dimensional meromorphic
vector bundle with connection $\nabla_0$ given by
\[
\nabla_{0 v_0}e_i=-b^j_ie_j
\]
for a fixed global frame $(e_1\ldots,e_n)$ of
$\mathscr{E}_0(X_0)$. The pullback to $X$ of
$(E_0,\Pi_0,\nabla_0)$ is (algebraically) defined by taking the
tensor product
\[
\mathscr{E}_0\otimes_{\mathscr{M}_0}\mathscr{M}
\]
and regarding it as a sheaf of differential $\mathscr{M}$-modules. In particular
\[
\mathscr{E}_0\otimes_{\mathscr{M}_0}\mathscr{M}(X)= \mathscr{E}_0(X_0)\otimes_{\mathbb{C}(X_0)}k.
\]
Therefore if $\nabla$ is the pullback of $\nabla_0$ then we have:
\[
\nabla_{\widetilde{v_0}}(e_i\otimes 1)=-b^j_i(e_j\otimes 1)
\]
where $\widetilde{v_0}$ stands for the lifting of $v_0$ to
$\mathscr{T}X$. Now if $v$ is another derivation in $\mathscr{T}X$
then there is an $f_v\in k$ such that $v=f_v\widetilde{v_0}$ and
so $\nabla_v=f_v\nabla_{\widetilde{v_0}}$. If we denote the basis
$e_i\otimes 1$ by $e_i$, then we have:
\[
\nabla_ve_i=-f_vb^j_ie_j
\]

\begin{rem}
If, as above, $\nabla$ is the pullback of $\nabla_0$ and
$(e_1,\ldots,e_n)$ is a cyclic basis for $\nabla_{0 v_0}$, i.e.
$\nabla_{0 v_0}e_i=e_{i+1}$ for $i\in\{1,\ldots,n-1\}$, then
$(e_1\otimes 1,\ldots,e_n\otimes 1)$ is a cyclic basis for
$\nabla_{\widetilde{v_0}}$ but not for $\nabla_v$. A cyclic basis
for $\nabla_v$ is $(e_1\otimes 1,f_ve_2\otimes 1,\ldots,
f_v^{n-1}e_n\otimes 1)$. Now, if $(y^i_je_i)_j$ is a full-system
of solutions for $\nabla_0$ then $(y^i_je_i\otimes 1)_j$ is a
full-system of solutions for $\nabla$, therefore a Picard-Vessiot
extension for $\nabla$ is generated by $(y^i_j)$ too.
\end{rem}

\subsection{Standard connections}

\begin{defn}
Let $k_0$ be a subfield of $k$. We say that $\nabla$ is \emph{defined over} $k_0$ if in (\ref{eq1})
\[
(a^i_j)\in M_{n\times n}(k_0)+\mathfrak{g}(k)
\]
\end{defn}

\begin{rem}
It follows from the definition and Lemma \ref{lem1} that, when $\nabla$ is such that $\det(y^i_j)=1$, the connection is defined over $X_0$.
\end{rem}

\begin{thm}
If $\nabla$ has reductive Galois group the connection is projectively equivalent to a connection defined over its projective Fano curve.
\end{thm}

\begin{proof} We fix the notation as above; in particular the horizontal sections of $\nabla$ satisfy the linear differential equation
\[
v(f^i)=a^i_jf^j.
\]
It is a well-known fact that if we tensor $\nabla$ with the connection given by the equation
\[
v(f)=-\frac{a^i_i}{n}f=-\frac{f}{n}\sum_ia^i_i
\]
then the resulting connection has unimodular Galois group \cite[Exercises 1.35.5]{SvdP}. In particular we may take $(y^i_j)$ such that $\det(y^i_j)=1$. Furthermore, \cite[Proposition 2.2]{SC} guarantees that the Galois group remains reductive after tensoring. We may therefore assume that $\nabla$ satisfies the hypotheses of Compoint's theorem. The result now follows from Lemma \ref{lem1}. \qed\end{proof}

\begin{rem}
In view of Proposition \ref{propfd}, the classification of connections with degenerate Fano curve is very simple: they correspond to extensions isomorphic to $G(k)/k$. Indeed, they are just the pullbacks of extensions $\mathbb{C}(G)/\mathbb{C}$, where $G$ is a connected algebraic group.
\end{rem}

\begin{defn}
We say that $\nabla$ is \emph{standard} if $K=\mathbb{C}(y^i_j)$ and $k=\mathbb{C}(X_0)$.
\end{defn}

\begin{cor}
Assume $\nabla$ has reductive Galois group and its projective Fano curve is not degenerate. The connection $\nabla$ is projectively equivalent to the pullback by a rational map of a standard connection over $X_0$ if and only if in Lemma \ref{lem1}
\[
(a^i_j)\in M_{n\times n}(\mathbb{C}(X_0)).
\]
In particular, if $G$ is finite then $\nabla$ is projectively equivalent to the pullback of a standard connection over $X_0$.
\end{cor}

\begin{proof} After tensoring like in the proof of the theorem we obtain the following diagram:
\[
\xymatrix{
 &  & K\\
\mathbb{C}(y^i_j)\ar@{-}[d]\ar@{^{(}->}[urr] & & k\ar@{-}[u] \\
\mathbb{C}(X_0)\ar@{^{(}->}[urr] & &
}
\]
If $\nabla$ is the pullback of a standard connection over $X_0$ then in Lemma \ref{lem1}, $a^\iota_\lambda$ is a solution to the system (\ref{syseq}) in $\mathbb{C}(X_0)$ and so $(a^i_j)-(b^i_j)=0$. Conversely $(a^i_j)\in M_{n\times n}(\mathbb{C}(X_0))$ defines a connection in a rank $n$ bundle over $X_0$ and since $K=k(y^i_j)$, we conclude that $\nabla$ is the pullback through the map $\mathbb{C}(X_0)\rightarrow k$.

The algebraic case follows from the fact that if $G$ is finite then $\mathfrak{g}=0$. \qed\end{proof}

\section{The classifying ruled surfaces.}

The relevance of having a classification of a collection of mathematical objects by other mathematical objects is that we can organize the former if we have a classification of the latter. In the case of standard equations this becomes a key point because, whereas the collection of standard equations for order two is classified by coverings of the sphere by the sphere, the collection of standard equations for higher order is infinite (see \cite{ber}) and so far no structure has been given to it. Furthermore, if an algorithmic implementation of the results in this paper is expected, as there is for the order two case, it would require a systematic way of classifying standard equations.

For the remainder of the article we will assume that the projective Fano curve $X_0$ is non-degenerate. Recall that this curve is defined by the homogeneous elements in $\mathbb{C}[X^i_j]^G$ vanishing when we evaluate them at $(y^i_j)$. The aim of this section is to explain how we can associate a ruled surface to each class of projectively equivalent standard connections over $X_0$.

\begin{defn}
A \emph{ruled surface} is a surface $\Sigma$ (a two dimensional $\mathbb{C}$-variety), together with a surjective map $\pi:\Sigma\rightarrow X_0$, where $X_0$ is a curve (a one dimensional $\mathbb{C}$-variety), such that the fibers $\pi^{-1}y$ are isomorphic to $\mathbb{P}^1(\mathbb{C})$, for every $y\in X_0$.
\end{defn}

Consider an irreducible standard connection $\nabla_0$ over $X_0$ with unimodular reductive Galois group $G$ and fundamental system of solutions $(y^i_j)$; and we fix $P(X^i_j)\in\mathbb{C}[X^i_j]^G$ such that $P(y^i_j)=1$. Let $P_1,\ldots,P_r\in\mathbb{C}[X^i_j]^G$ be homogeneous generators of the subalgebra of $G$-invariants in $\mathbb{C}[X^i_j]$. We denote by $I^G$ the kernel of the evaluation $\mathbb{C}$-morphism
\begin{eqnarray*}
\Psi\upharpoonright_{\mathbb{C}[X^i_j]^G} : \mathbb{C}[X^i_j]^G & \longrightarrow & \mathbb{C}(X_0)\\
      P_l(X^i_j)    & \longmapsto     & P_l(y^i_j) \quad \forall l\in\{1,\ldots,r\},
\end{eqnarray*}
which includes the homogeneous ideal $J$ that defines the Fano curve. Geometrically we have a curve $V(I^G)$ in $(\mathbb{C}^{n\times n})^G$ generating the cone $V(J)$.

We embed $\mathbb{C}^{n\times n}$ into $\mathbb{P}^{n\times n}(\mathbb{C})$ by introducing the homogeneous coordinates $(z:x^i_j)$ for $\mathbb{P}^{n\times n}(\mathbb{C})$ and identifying $\mathbb{C}^{n\times n}$ with $z=1$. In particular we put $X^i_j=\frac{x^i_j}{z}$. We extend the action of $G$ on $\mathbb{C}^{n\times n}$ to $\mathbb{P}^{n\times n}(\mathbb{C})$ by declaring
\[
(g^i_j):z\mapsto z,\ x^i_j\mapsto x^i_lg^l_j.
\]
Now consider $\mathbb{P}^{n\times n}(\mathbb{C})\times \mathbb{P}^{n\times n-1}(\mathbb{C})$, where the second factor has homogeneous coordinates $(w^i_j)$. Again, we extend the action of $G$ by declaring that on the second factor we have
\[
(g^i_j):w^i_j\mapsto w^i_l g^l_j.
\]
So that the variety $Y$ defined by the homogeneous equations $x^i_j w^\iota_\kappa-x^\iota_\kappa w^i_j$ is the blown-up of $\mathbb{P}^{n\times n} (\mathbb{C})$ at $0$ and it is invariant under the $G$-action. Indeed,
\[
(g^i_j)(x^i_j w^\iota_\kappa- x^\iota_\kappa w^i_j)= g^l_jg^\lambda_\kappa(x^i_l w^\iota_\lambda-x^\iota_\lambda w^i_l).
\]
So we have a commutative diagram
\[
\xymatrix{
Y\ar@{^{(}->}[rr]\ar[rrd] & & \mathbb{P}^{n\times n}(\mathbb{C})\times \mathbb{P}^{n\times n-1}(\mathbb{C})\ar[d]\\
                          & & \mathbb{P}^{n\times n}(\mathbb{C})
}
\]
of $G$-morphisms which yields
\[
\xymatrix{
Y^G\ar@{^{(}->}[rr]\ar[rrd]_\varpi & & (\mathbb{P}^{n\times n}(\mathbb{C})\times \mathbb{P}^{n\times n-1}(\mathbb{C}))^G\ar[d]\\
                          & & (\mathbb{P}^{n\times n}(\mathbb{C}))^G
}
\]
We set $\Sigma:=\varpi^{-1}\overline{V(J)}$.

\begin{lem} The map:
\begin{eqnarray*}
\pi:\Sigma & \longrightarrow & X_0\\
 (z:x^i_j, w^\iota_\kappa)\cdot G & \longmapsto & (w^\iota_\kappa)\cdot G
\end{eqnarray*}
defines a ruled surface.
\end{lem}

\begin{proof} Let $(q^\iota_\kappa)\cdot G\in X_0$. Let $R(X^i_j)\in\mathbb{C}[X^i_j]^G$ be homogeneous and such that $R(q^\iota_\kappa)\ne 0$.  If $Q(X^i_j)\in \mathbb{C}[X^i_j]^G$ is such that $Q(q^\iota_\kappa)=0$, and if $(p^i_j)\cdot G\in (\mathbb{C}^{n\times n})^G$ is such that $(z:p^i_j,q^\iota_\kappa)\in\Sigma$, then
\[
Q(p^i_j)R(q^\iota_\kappa)-R(p^i_j)Q(q^\iota_\kappa)=0,
\]
implying that $Q(p^i_j)=0$. Therefore $(p^i_j)\cdot G=(q^\iota_\kappa)\cdot G$. We have then $\varpi^{-1}(q^\iota_\kappa)=\overline{\{(z:q^i_j,q^\iota_\kappa)\cdot G|\ z\in\mathbb{C}\}}$, so that $\varpi^{-1}(q^\iota_\kappa)$ is isomorphic to $\mathbb{P}^1(\mathbb{C})$.\qed\end{proof}

\begin{rem}
If follows from the proof that the ruled surface $\Sigma$ is determined by the homogeneous polynomials
\[
Q(x^i_j)R(w^\iota_\kappa)-R(x^i_j)Q(w^\iota_\kappa),\quad Q(X^i_j)\in J,\ R(X^i_j)\in\mathbb{C}[X^i_j]^G.
\]
In particular two projectively equivalent standard connections have the same ruled surface because their ideals $J$ coincide. Now we will prove that the connection is uniquely determined then by $\pi:\Sigma\rightarrow X_0$ together with the curve $\varpi^{-1}(\overline{V(I_G)})$.
\end{rem}

\begin{lem}\label{rslem}
To a standard connection $\nabla_0$ over $X_0$ with unimodular reductive Galois group we associate a ruled surface $\pi:\Sigma\rightarrow X_0$. If two standard connections are projectively equivalent then their associated ruled surfaces coincide. Among the connections associated to $\pi:\Sigma\rightarrow X_0$, $\nabla_0$ is characterized by the curve $\varpi^{-1}(\overline{V(I_G)})$.
\end{lem}

\begin{proof} The first two statements in the lemma follows from the previous one and the remark just above. Now if we restrict ourselves to the open set $U=\{z\ne 0\}$ on $\mathbb{P}^{n\times n}(\mathbb{C})$ and to the open set $U_0=\{P(q^\iota_\kappa)\ne 0\}$ on $X_0$, then the map (using the notation from Compoint's Theorem, Theorem \ref{compthm})
\begin{eqnarray*}
\mathbb{C}[X^i_j]^G & \longrightarrow & \mathbb{C}(X_0)\\
        P_i         & \longmapsto     & f_i
\end{eqnarray*}
sends, on the level of varieties, $U_0$ into $V(I_G)$. This map is the one defining the maximal differential ideal $I$ in Compoint's Theorem.\qed\end{proof}

\begin{rem}
The ruled surfaces obtained by blowing-up the vertex of a cone can also be obtained by taking the projective bundle defined by a rank-two holomorphic vector bundle over the base curve $X_0$. In order to obtain $\Sigma$ through a vector bundle, we are going to construct a two-dimensional vector bundle over $\mathbb{P}^{n\times n-1}(\mathbb{C})^G$ which we will then pullback to $X_0$. We adapt the exposition in \cite[Example V.2.11.4]{hart} to our specific setting.\\
Given a positive integer $N$ we denote by $\mathbb{C}[X^i_j]^G_{\ge N}$ the $\mathbb{C}[X^i_j]^G$-algebra of polynomials of degree greater than or equal to $N$. Note that
\[
\mathbb{C}[X^i_j]^G=\bigoplus_{m\in\mathbb{N}}\mathbb{C}[X^i_j]^G_m
\]
where $\mathbb{C}[X^i_j]^G_m$ denotes the homogeneous polynomial of degree $m$ and
\[
\mathbb{C}[X^i_j]^G_{[N]}=\bigoplus_{m\in\mathbb{N}}\mathbb{C}[X^i_j]^G_{mN}
\]
can both be used as homogeneous coordinates of $\mathbb{P}^{n\times n-1}(\mathbb{C})^G$.

Now assume $P_1,\ldots, P_r\in\mathbb{C}[X^i_j]^G$ is a minimal set of homogeneous generators of $\mathbb{C}[X^i_j]^G$. Denote by $n_i$ the degree of $P_i$, $i\in\{1,\ldots,r\}$, and by $N=[n_1,\ldots,n_r]$ the least common multiple of these degrees. Set $N=N_in_i$
\[
M_0=\mathbb{C}[X^i_j]^G\cdot Z\simeq\mathbb{C}[X^i_j]^G
\]
and
\[
M_N=\sum_{m_1n_1+\ldots+m_rn_r=N}\mathbb{C}[X^i_j]^G\prod_lP_l^{m_l}\simeq\mathbb{C}[X^i_j]^G_{\ge N}.
\]
Each of these two free graded $\mathbb{C}[X^i_j]^G$-modules define a rank-one holomorphic vector bundles $L_0\rightarrow \mathbb{P}^{n\times n-1}(\mathbb{C})^G$ and $L_N\rightarrow \mathbb{P}^{n\times n-1}(\mathbb{C})^G$ respectively. Let $V\rightarrow \mathbb{P}^{n\times n-1}(\mathbb{C})^G$ be the fibred sum of $L_0$ and $L_N$. Denote by $\mathscr{O}$ the sheaf of holomorphic sections of $L_0\rightarrow \mathbb{P}^{n\times n-1}(\mathbb{C})^G$ and by $\mathscr{O}$(N) the sheaf of holomorphic sections of $L_N\rightarrow \mathbb{P}^{n\times n-1}(\mathbb{C})^G$. In particular we have that the sheaf of sections of $V\rightarrow \mathbb{P}^{n\times n-1}(\mathbb{C})^G$ is $\mathscr{V}=\mathscr{O}\oplus\mathscr{O}(N)$.

We want to construct the graded $\mathbb{C}[X^i_j]^G$-algebra given by the symmetric algebra of $V$ (cf. \cite[Section II.7]{hart}). Set
\[
\mathbb{C}[X^i_j]^G[Z]\otimes_{\mathbb{C}[X^i_j]^G} \mathbb{C}[X^i_j]^G[\prod_lP_l^{m_l}|\ m_1n_1+\ldots+m_rn_r=N]=:\mathbb{C}[X^i_j]^G[Z][Y^i_j]^G_{\ge N}
\]
where $Z$ will have degree $1$ and $\mathbb{C}[Y^i_j]^G_{\ge N}$ will be graded by $\deg-(N-1)$. To obtain the desired algebra we take the quotient algebra defined by the $\mathbb{C}[X^i_j]^G$-morphism
\begin{eqnarray*}
\mathbb{C}[X^i_j]^G[Z][Y^i_j]^G_{\ge N}        & \longrightarrow & \mathbb{C}[X^i_j]^G\\
        Z                                      & \longmapsto     & 1\\
        Q(Y^i_j)                               & \longmapsto     & Q(X^i_j),\quad\textrm{for }Q(Y^i_j)\in\mathbb{C}[Y^i_j]^G_{\ge N}.
\end{eqnarray*}
The biggest homogeneous ideal in the kernel of this morphism is generated by the elements
\[
Q(X^i_j)R(Y^i_j)-R(X^i_j)Q(Y^i_j)\quad Q(X^i_j),R(X^i_j)\in \mathbb{C}[X^i_j]^G_{\ge N},
\]
which would yield the same ideal defining $Y^G$ if instead of using $\mathbb{C}[X^i_j]^G$ as homogeneous coordinate ring of $\mathbb{P}^{n\times n-1}(\mathbb{C})^G$ we use $\mathbb{C}[X^i_j]^G_{[N]}$. So we conclude that $Y^G$ corresponds to the projective bundle defined by $V\rightarrow \mathbb{P}^{n\times n-1}$.
\end{rem}

\begin{rem}
 Let us denote the pullback to $X_0$ of $\mathscr{V}$ (resp. of $\mathscr{O}$, of $\mathscr{O}(N)$) by $\mathscr{V}_{X_0}$ (resp. by $\mathscr{O}_{X_0}$, by $\mathscr{O}_{X_0}(N)$). Then as $Y^G$ is obtained as the projective bundle from $\mathscr{V}$, the portion $\Sigma$  over $X_0$ is obtained by taking the projective bundle defined by $\mathscr{V}_{X_0}$. In particular, since $\mathscr{V}_{X_0}=\mathscr{O}_{X_0}\oplus\mathscr{O}_{X_0}(N)$ and $\mathscr{O}_{X_0}$ is the sheaf of holomorphic functions over $X_0$, we conclude that $\Sigma$ is uniquely determined by $\mathscr{O}_{X_0}(N)$.\\
It is costumary to normalize $\mathscr{V}$ by tensoring with $\mathscr{O}(-N)=\textrm{Hom}_\mathscr{O}(\mathscr{O}(N),\mathscr{O})$. This does not change the induced projective bundle, so that the rank two vector bundle inducing $\Sigma$ is $\mathscr{O}_{X_0}\oplus\mathscr{O}_{X_0}(-N)$.
\end{rem}

\begin{thm}
Let $\nabla_0$ and $\nabla$ be two standard connections, defined on the same meromorphic vector bundle $E\rightarrow X_0$, with reductive Galois groups such that their projective Galois groups are isomorphic. Then $\nabla_0$ and $\nabla$ are projectively equivalent if and only if their associated ruled surface coincide.
\end{thm}

\begin{proof} The necessity has been already established in Lemma \ref{rslem}. We retain the same notation as above for the connection $\nabla_0$. In particular $\pi:\Sigma\rightarrow X_0$ is the associated ruled surface, and $P(y^i_j)=1$ for some fixed $P(X^i_j)\in\mathbb{C}[X^i_j]^G$. Because $\overline{V(J)}$ is a projective variety of dimension two, we have that the irreducible variety $\overline{V(I_G)}$ is in the closure of the curve
$V(J)\cap V(P(X^i_j)-1)$, and they coincide if $V(J)\cap V(P(X^i_j)-1)$ is irreducible. Without loss of generality we may assume that $\nabla_0$ and $\nabla$ are normalized so that their $n$-th exterior product has rational sections. In such a case $P(X^i_j)=\det(X^i_j)$, $P(X^i_j)-1$ is prime and
$\overline{V(I_G)}=\overline{V(J)\cap V(P(X^i_j)-1)}$. It follows that the curve in $\Sigma$ defining $\nabla_0$ and $\nabla$ coincide (cf. Lemma \ref{rslem}), whence $\nabla=\nabla_0$.\qed\end{proof}

\subsection*{Computing the ruled surfaces}

Although the proof above helps us characterizing the ruled
surfaces arising from standard equations, the proof offers little
towards effectively computing these. In fact, even if we manage to
get the $N$ in the expression
$\mathscr{V}=\mathscr{O}\oplus\mathscr{O}(N)$, it is not easy to
obtain the pullback and the blow up to explicitly get
$\mathscr{V}_{X_0}=\mathscr{O}_{X_0}\oplus\mathscr{O}_{X_0}(N)$.
To carry out the computations we will rely on the Nash blowing-up
(cf. \cite{nobile}).

\begin{defn}
Let $X$ be a subvariety of dimendion $r$ of a complex variety of dimension
$m$. Let $S$ be the set of singular
points of $X$ and $X_S$ its complement in $X$. Set
\begin{eqnarray*}
\eta :X_S & \longrightarrow & X\times G^m_r\\
       x  & \longmapsto     & (x,T_{X,x})
\end{eqnarray*}
where $G^m_r$ is the Grassmanian of $r$-planes in $m$-space and
$T_{X,x}$ is the tangent to $X$ at $x$ seen as an $r$-plane inside
the tangent to the ambient variety at $x$. The Nash blow-up is
\[
\varpi: X^*\longrightarrow X
\]
where $X^*$ is the closure of $\eta(X_S)$ and $\varpi$ is given by
the first projection.
\end{defn}

\begin{rem}
Nash blowing-up is a monoidal transformations \cite[Theorem 1]{nobile}, therefore it transforms the variety $X$ into a bi-rationally equivalent space \cite[Proposition 7.16]{hart}. The same is true for the blow-up at a point. Thus, the spaces obtained from the cone defined by $(y^i_j)$ via Nash blowing-up or one point blow-up are birationally equivalent.
\end{rem}

Let $P(X^i_j)$ be a homogeneous function of degree $d$ such that $P(y^i_j)=0$, then
\[
v\big(P(y^i_j)\big)=v(y^\iota_\kappa)\frac{\partial P}{\partial X^\iota_\kappa}(y^i_j)=0;
\]
and, from Euler's Theorem
\[
dP(y^i_j)=y^\iota_\kappa\frac{\partial P}{\partial X^\iota_\kappa}(y^i_j)=0.
\]
Therefore at the regular points of the cone defined by $(y^i_j)$, the tangent space is spanned by $(y^i_j)$ and $(v(y^i_j))$. From where we obtain:

\begin{prop}\label{compruled}
Let $\nabla$ be a standard connection with reductive and unimodular Galois group $G$, and let $P_0,\ldots,P_r$ be a set of homogeneous generators of $\mathbb{C}[X^i_j]^G$. If $(y^i_j)$ is a fundamental system of solutions, then the ruled surface associated to $\nabla$ corresponds to the projective bundle defined by the rank-two vector bundle $\mathscr{L}\oplus\mathscr{L}'$ where $\mathscr{L}$ is defined by the map (cf.\cite[Proposition 7.1]{hart})
\begin{eqnarray*}
\mathbb{C}[X^i_j]^G & \longrightarrow & k\\
P_i(X^i_j) & \longmapsto & P_i(y^i_j)
\end{eqnarray*}
and $\mathscr{L}'$ is defined by
\begin{eqnarray*}
\mathbb{C}[X^i_j]^G & \longrightarrow & k\\
P_i(X^i_j) & \longmapsto & P_i\big(v(y^i_j)\big)
\end{eqnarray*}
\end{prop}

\begin{proof}
Identifying all the tangent spaces to $\mathbb{C}^{n\times n}$ with the tangent at the origin, $G$ acts on the tangent spaces in the same way as it acts on $\mathbb{C}^{n\times n}$. This shows the map $\eta$ in the definition of the Nash blowing-up as a $G$-morphism. Taking quotients by the $G$-action, the $2$-space spanned by $(y^i_j)$ and $(v(y^i_j))$ maps to the space spanned by $(y^i_j)\cdot G$ and $(v(y^i_j))\cdot G$.
\end{proof}

\section*{Appendix: The algebraic case}

\subsection*{Computing the genus of the projective curve defined by $(y^i_j)$.}
\begin{rem}
The content of this part of the appendix is a generalization of the presentation of \cite[Lemma 1.5]{baldassarri} to higher order.
\end{rem}

Given a $n$-th order irreducible linear differential equation $L(y)=0$ over $X$ with algebraic solutions, and a point $p\in X$, we denote by
\[
E(L,p)=\{\alpha_{1,p},\alpha_{2,p},\ldots,\alpha_{n,p}\}
\]
the collection of generalized exponents of $L(y)=0$ at $p$ (i.e. the roots of the indicial polynomial, which are all rational because the solutions are algebraic) ordered so that $\alpha_{i,p}<\alpha_{j,p}$ if $i<j$. We set
\[
\Delta(L,p)=\alpha_{n,p}-\alpha_{1,p}-(n-1)
\]
and we let $e(L,p)$ be the least common denominator of $E(L,p)-E(L,p)=\{a-b|\ a,b\in E(L,p)\}$ (i.e. the smallest $m\in\mathbb{N}$ such that $m(E(L,p)-E(L,p))\subseteq\mathbb{Z}$. If $S\subseteq X$ we set
\[
\Delta(L,S):=\sum_{p\in S}\Delta(L,p).
\]
As $E(L,p)=\{0,1,\ldots,n-1\}$ for almost all $p\in X$, we have $\Delta(L,p)=0$ for all but finitely many $p\in X$, and for a sufficiently large $S$, the number $\Delta(L,S)$ attains limiting value $\Delta(L)$.

\begin{lem}
Let $f:X\rightarrow X_0$ be a morphism of compact Riemann surfaces of degree $M$, and assume that $L(y)=0$ is the pullback of $L_0(y)=0$ through $f$. Furthermore, assume all solutions to $L_0(y)=0$ are algebraic. Then, if $g$ (resp. $g_0$) denotes the genus of $X$ (resp. of $X_0$), we have:
\[
M(\frac{\Delta(L_0)}{n-1}-2(g_0-1))=\frac{\Delta(L)}{n-1}-2(g-1).
\]
\end{lem}

\begin{proof} Let $S_0\in X_0$ be a finite collection of points containing all ramifications of $f$ and all singularities of $L_0(y)=0$, and set $S:=f^{-1}(S_0)$. So, if $e_{p_0}$ denotes the ramification index of $p_0$ in $f$, then
\begin{eqnarray*}
\Delta(L,f^{-1}(p_0)) & = & \sum_{p|p_0}\Delta(L,p)\\
                      & = & \sum_{p|p_0}\alpha_{n,p}-\alpha_{1,p}-(n-1)\\
                      & = & \sum_{p|p_0}e_{p_0}(\alpha_{n,p_0}-\alpha_{1,p_0})-(n-1))\\
                      & = & \frac{M}{e_{p_0}}(e_{p_0}(\alpha_{n,p_0}-\alpha_{1,p_0})-(n-1))\\
                      & = & M(\Delta(L_0,p_0)+(n-1))-\textrm{Card}(f^{-1}(p_0))(n-1).
\end{eqnarray*}
Thus $\Delta(L,f^{-1}(p_0))+(n-1)\textrm{Card}(f^{-1}(p_0))=M(\Delta(L_0,p_0)+(n-1))$ and
\[
\Delta(L,S)+(n-1)\textrm{Card}(S)=M(\Delta(L_0,S_0)+(n-1)\textrm{Card}(S_0))
\]
As $S_0$ contains all ramifications of $f$, the Hurwitz genus formula implies
\[
2(g-1)-2M(g_0-1)=M\ \textrm{Card}(S_0)-\textrm{Card}(S).
\]
Combining the last two equalities we obtain the desired conclusion by noticing that $S_0$ contains all singularities of $L_0(y)=0$.\qed\end{proof}

\begin{prop}\label{lemhur}
Under the hypotheses of the lemma we have:
\[
\sum_{p_0\in S_0}\left(\frac{1}{e(L_0,p_0)}-1\right)=2(g_0-1)-\frac{2(g-1)}{M}.
\]
\end{prop}

\begin{proof} Suppose $f$ corresponds to the field extension
\[
\mathbb{C}(X_0)\subseteq\mathbb{C}(X_0)[\frac{y_1}{y_n},\ldots,\frac{y_{n-1}}{y_n}],
\]
where $y_1,\ldots,y_n$ denote a full system of solutions of $L_0(y)=0$. Then $e_{p_0}=e(L_0,p_0)$, and it follows that
\begin{eqnarray*}
\frac{1}{n-1}(M\Delta(L_0)-\Delta(L)) & = & \frac{1}{n-1}\Big\{ M\sum_{p_0\in S_0}(\alpha_{n,p_0}-\alpha_{1,p_0}-(n-1))\\
                                      &   & \qquad -M\sum_{p_0\in S_0}(\alpha_{n,p_0}-\alpha_{1,p_0}-\frac{n-1}{e(L_0,p_0)})\Big\}\\
                                      & = & \frac{M}{n-1}\sum_{p_0\in S_0}\left(\frac{n-1}{e(L_0,p_0)}-(n-1)\right)\\
                                      & = & M\sum_{p_0\in S_0}\left(\frac{1}{e(L_0,p_0)}-1\right).
\end{eqnarray*}
The statement now follows from the previous lemma. \qed\end{proof}

\begin{exa}
Consider the three following equations: Ulmer's $G_{54}$ equation ~\cite{ulmer}
{\footnotesize
\[
y'''+\frac{3(3x^2-1)}{x(x-1)(x+1)}y''+\frac{221x^4-206x^2+5}{12x^2(x-1)^2(x+1)^2}y'+\frac{374x^6-673x^4+254x^2+5}{54x^3(x-1)^3(x+1)^3}y=0
\]
}
with singular points at $0$, $1$, $-1$ and $\infty$, with respective exponents
\[
\{-\frac{1}{6},\frac{1}{3},\frac{4}{3}\},\quad\{-\frac{1}{6},\frac{1}{3},\frac{4}{3}\}
\quad\{-\frac{1}{6},\frac{1}{3},\frac{4}{3}\},\quad\{-\frac{1}{6},\frac{1}{3},\frac{4}{3}\};
\]
the Geiselmann-Ulmer $F_{36}^{SL_3}$ equation ~\cite{Geis}
\[
y'''+\frac{5(9x^2+14x+9)}{48x^2(x+1)^2}y'-\frac{5(81x^3+185x^2+229x+81)}{432x^3(x+1)^3}y=0
\]
with singular points at $0$, $1$ and $\infty$, with respective exponents
\[
\{1,\frac{3}{4},\frac{5}{4}\},\quad \{\frac{5}{6},\frac{11}{6},\frac{1}{3}\},\quad \{-1,\frac{-3}{4},\frac{-5}{4}\};
\]
and the equation
\[
y'''+\frac{1}{48}\frac{41z^2-50z+45}{(z-1)^2z^2}y'-\frac{1}{432}\frac{364z^3-665z^2+1030z-405}{(z-1)^3z^3}y=0
\]
with singularities at $0$, $1$ and $\infty$ with respective exponents
\[
\{\frac{3}{4},1,\frac{5}{4}\},\quad\{\frac{1}{2},1,\frac{3}{2}\},\quad\{-\frac{4}{3},-\frac{13}{12},-\frac{7}{12}\}.
\]
For the first equation we have $g_0=0$, $M=|PG_{54}|=18$ and $e_0=e_1=e_{-1}=e_\infty=2$, so that
\[
\sum_{i\in\{0,1,-1,\infty\}}\left(1-\frac{1}{e_i}\right)=2;
\]
for the second equation we have $g_0=0$, $M=|F_{36}|=36$ and $e_0=e_\infty=4$, $e_1=2$ so that
\[
\sum_{i\in\{0,1,\infty\}}\left(1-\frac{1}{e_i}\right)=2;
\]
and, for the third equation we again have $g_0=0$, $M=|F_{36}|=36$ and $e_0=e_\infty=4$, $e_1=2$. This tells us that the algebraic extension given by the ratio of solutions is a curve of genus $1$. Indeed, the three equations are related to the generalized hypergeometric equation defining ${}_3F_2(-\frac{1}{12},\frac{1}{6},\frac{2}{3},\frac{1}{2},\frac{3}{4};z)$, i.e.
{\small
\[
y'''+\frac{3}{4}\frac{5z-3}{z(z-1)}y''+\frac{1}{24}\frac{43z-9}{z^2(z-1)}y'-\frac{1}{108z^2(z-1)}y=0.
\]
}
The first of our equations is projectively equivalent to the pullback of this one by the map $z(x)=\frac{1}{16}\frac{(x^2+1)^4}{x^2(x+1)^2(x-1)^2}$; the second one is projectively equivalent to the pullback by the map $z(x)=\frac{4(x-1)}{x^2}$; and the third one is the normalized form which is standard. If one takes three linearly independent solutions to each of these equations, we can see that they satisfy a homogeneous equation with coefficients in $\mathbb{C}$ of degree $3$ in three variables defining an elliptic curve.
\end{exa}

\begin{exa}
Consider the following two equations. The first is van Hoeij's $H_{72}^{SL_3}$ equation ~\cite{ulmer}, i.e.
\begin{eqnarray*}
0 & = & y'''+\frac{21x^2-24x-1}{(3x^2+1)(x-1)}y''+\frac{1}{48}\frac{4437x^3-5973x^2+171x-683}{(3x^2+1)^2(x-1)}y' \\
  &   & \qquad +\frac{1}{216}\frac{13338x^4-22647x^3+1983x^2-7297x-737}{(3x^2+1)^3(x-1)}y.
\end{eqnarray*}
The singular points are $1$ (which actually is an apparent singularity), $\frac{i\sqrt{3}}{3}$, $-\frac{i\sqrt{3}}{3}$ and $\infty$, with respective exponents
\[
\{0,1,3\},\quad\{-\frac{7}{12},-\frac{1}{3},-\frac{1}{12}\},
\quad\{-\frac{7}{12},-\frac{1}{3},-\frac{1}{12}\},\quad\{\frac{13}{12},\frac{4}{3},\frac{19}{12}\}.
\]
The second is
{\small
\[
y'''+\frac{1}{432}\frac{405z^2-469z+384}{(z-1)^2z^2}-\frac{1}{11664}\frac{10935z^3-18803z^2+27196z-10368}{(z-1)^3z^3}=0.
\]
}
Here the singular points are $0$, $1$ and $\infty$, with respective exponents
\[
\{\frac{2}{3},1,\frac{4}{3}\},\quad\{\frac{5}{9},\frac{8}{9},\frac{14}{9}\},\quad\{-\frac{5}{4},-1,-\frac{3}{4}\}.
\]
For the first equation we have $g_0=0$, $M=|H_{72}|=72$ and $e_\frac{i\sqrt{3}}{3}=e_{-\frac{i\sqrt{3}}{3}}=e_\infty=4$, so that
\[
\sum_{j\in\{\frac{i\sqrt{3}}{3},-\frac{i\sqrt{3}}{3},\infty\}}\left(1-\frac{1}{e_j}\right)=\frac{9}{4}=\frac{2(10-1)}{M}+2.
\]
For the second equation we have $g_0=0$, $M=|H_{216}|=216$ and $e_0=e_1=3$, $e_\infty=4$, so that
\[
\sum_{i\in\{0,1,\infty\}}\left(1-\frac{1}{e_i}\right)=\frac{25}{12}=\frac{2(10-1)}{M}+2.
\]
This tells us that the algebraic extension given by the ratio of solutions is a curve of genus $10$. Indeed, the two equations are related to the generalized hypergeometric equation ${}_3F_2(-\frac{1}{36},\frac{2}{9},\frac{17}{36},\frac{1}{3},\frac{2}{3};z)$, i.e.
\[
y'''+\frac{1}{3}\frac{11z-6}{z(z-1)}y''+\frac{1}{432}\frac{-96+757z}{z^2(z-1)}y'-\frac{17}{5832}\frac{1}{z^2(z-1)}=0.
\]
The first of our equations is projectively equivalent to the pullback of this one by the map $z(x)=\frac{1}{2}\frac{(x+1)^3}{(1+3x^2)}$; whereas the second one is the normalized form and is standard. If one takes three linearly independent solutions to each of these equations, we can see that they satisfy a homogeneous equation of degree $6$ in three variables defining a curve of genus $10$.
\end{exa}

\subsection*{The ruled surfaces for the second order standard equations}

The standard equation $St_G$ of second order for the Galois group $G$ is \cite{ber}
\[
y''+\left(\frac{a}{x^2}+\frac{b}{(x-1)^2}+\frac{c}{x(x-1)}\right)y=0
\]
where $a=\frac{1-\lambda^2}{4}$, $b=\frac{1-\mu^2}{4}$, and $c=\frac{\lambda^2+\mu^2-1-\nu^2}{4}$; and, $(\lambda,\mu,\nu)$ is $(1/3,1/2,1/3)$ if $G=A_4$, $(1/3,1/2,1/4)$ if $G=S_4$, $(1/3,1/2,1/5)$ if $G=A_5$, and $(1/2,1/2,1/n)$ if $G=D_{2n}$.
Now, if $(y_1,y_2)$ are the solutions to the equation
\[
y''-fy=0,
\]
then $(y'_1,y'_2)$ are the solutions to the equation
\[
y''-\frac{f'}{f}y'-fy=0.
\]
In this way we can easily compute $\mathscr{L}$ and $\mathscr{L}'$ from Proposition \ref{compruled} using the algorithm in \cite{We}.

\begin{exa}
For $St_{A_4}$ the invariants of degree $24$ are spanned by
\[
x^8(x-1)^8,\quad x^8(x-1)^7,\quad x^9(x-1)^6, \textrm{ and } x^8(x-1)^6
\]
so $\mathscr{L}$ is generated by $1$, $x$, $x-1$, $(x-1)^2$, hence $\mathscr{L}=\mathscr{O}(2)$. On the other hand the equation $St'_{A_4}$ with solutions the derivative of the solutions to $St_{A_4}$ has invariants of degree $24$ spanned by
\[
\frac{f_6^4}{(x-1)^{16}x^{16}},\quad \frac{f_6^2f_{12}}{(x-1)^{17}x^{16}},\quad \frac{f_8^3}{(x-1)^{18}x^{15}}, \textrm{ and } \frac{f_{12}^2}{(x-1)^{18}x^{16}}
\]
where $f_l$, $l\in\{6,8,12\}$, is a polynomial of degree $l$; so $\mathscr{L}'$ is generated by $1$, $x\frac{f_8^3}{f_{12}^2}$, $(x-1)\frac{f_6^2}{f_{12}}$, $(x-1)^2\frac{f_6^4}{f_{12}^2}$, hence $\mathscr{L}'=\mathscr{O}(26)$. The ruled surface corresponding to $A_4$ is the projective bundle defined by $\mathscr{O}(2)\oplus\mathscr{O}(26)$, which is the same as the one defined by $\mathscr{O}\oplus\mathscr{O}(24)$.
\end{exa}

Similar computations shows that:
\begin{itemize}
\item for $A_4$ the ruled surface is $\mathbb{P}\Big(\mathscr{O}(2)\oplus\mathscr{O}(26)\Big)$;
\item for $S_4$ the ruled surface is $\mathbb{P}\Big(\mathscr{O}(1)\oplus\mathscr{O}(25)\Big)$;
\item for $A_5$ the ruled surface is $\mathbb{P}\Big(\mathscr{O}(1)\oplus\mathscr{O}(61)\Big)$;
\item for $D_{2n}$ the ruled surface is $\mathbb{P}\Big(\mathscr{O}(2)\oplus\mathscr{O}(2[2n+1])\Big)$, if $2\!\not|n$; and,
\item for $D_{2n}$ the ruled surface is $\mathbb{P}\Big(\mathscr{O}(1)\oplus\mathscr{O}(2n+1)\Big)$, if $2|n$.
\end{itemize}

\end{document}